\documentclass{myart2}

\begin{document}
	
	\title{A note on the definition of good special flows of Kanigowski and Ravotti}
	\author{Przemys{\l}aw Kucharski\footnote{
			This research was funded by the Flagship Project Central European Mathematical Research Lab, under the program Excellence Initiative -- Research University at the Jagiellonian University in Krak\'ow. ORCiD: 0000-0002-3826-5827 email: przemyslaw.kucharski@uj.edu.pl
	}}
	\affil{Jagiellonian University, Faculty of Mathematics and Computer Science,\\
		Prof. St. {\L}ojasiewicza St 6, PL30348, Cracow, Poland}
	
	\maketitle
	
	
	\begin{abstract}
		In this short note we show that certain definition from the work of Kanigowski and Ravotti \cite{kanigowski2024multiplemixingparabolicsystems} can be relaxed. More specifically, we show that the definition of good special flows, that is, flows exhibiting a form of mechanism of mixing, shearing of Birkhoff sums, from \cite{kanigowski2024multiplemixingparabolicsystems} can be relaxed. Specifically, we show that the disintegration corresponding to almost partitions, which are uniformly sheared along the direction of the flow, of the base can depend on $k$-tuple of times. This note serves as a complementary material to the article \cite{kucharski:multmixing}, where we use the new definition of good special flows to show that flows of Fayad \cite{fayad2002} are mixing of any order.
	\end{abstract}
	\setcounter{tocdepth}{2}
	\tableofcontents
        
	\section{Introduction}
    In \cite{kanigowski2024multiplemixingparabolicsystems} Kanigowski and Ravotti introduced locally uniformly shearing flows, LUS flows for short, and showed that whenever a LUS flow is mixing, it necessarily is mixing of any order. They also provide effective conditions for special flows over $\T$ to be LUS. A step in the proof involves introducing the so called good special flows. It is shown that good special flows are LUS. Intuitevly, good special flows are special flows over a measurable space that by definition exhibit a disintegration of the base space and an almost partition $\cc{P}_{\tb}$ (depending on a $(k+1)$-tuple of times $\tb$) of the base space into subsets $I\in \cP$ that respects the disintegration, and such that the pushed elements $f_t(I)$ are uniformly stretched for some $t\in\tb$. While good special flows are defined over an arbitrary measurable space, thus allowing to consider special flows over $\T^k$ for $k>1$, the definition is not general enough to cover the case when we have shearing of Birkhoff sums only when restricted to segments in transverse directions that alternate with $t\in\tb$. In this short note we show that a variation of the definition of good special flows that allows the disintegration to vary according to the time $t\geq 0$ still gives LUS flows. The said variation is crucial in the proof of multiple mixing for Fayad flows \cite{fayad2002}, see \cite{kucharski:multmixing}. The proof the good special flows with varying disintegration are LUS is analogous to the proof of Kanigowski and Ravotti that good special flows are LUS.
    \begin{thmmain}\label{thmmain:goodspecialflows->lus}
		Good special flows with varying disintegration are LUS.
	\end{thmmain}
        
	\section{Preliminaries}\label{section:preliminaries}
	\subsection{Measure theory}
	In this section we assume that $(X, \mathcal{B}, \mu)$ is a Lebesgue measure space and $(X, \dist)$ is a complete metric space. Let us recall a few notions from the measure theory needed for the formulation of good special flows in Definition \ref{definition:goodspecialsflows}. This section is based on the work of Ravotti and Kanigowski \cite{kanigowski2024multiplemixingparabolicsystems}.
	\begin{definition}
		Any pairwise disjoint family $\cP$ of measurable subsets of $X$ will be called a partial partition.
	\end{definition}
	\begin{definition}[Almost partitions]
		Let $\varepsilon \geq 0$ and let $A \subseteq X$ be measurable. A family $\cP$ of measurable subsets of $X$ is an \emph{$\varepsilon$-almost partition of $A$} if it contains pairwise disjoint sets and 
		\[
		\mu \left(A \, \triangle \, \bigcup \{P \in \cP\colon P \cap A \neq \emptyset\}\right) <\varepsilon \mu(A).
		\]
	\end{definition}
	
	\subsection{Flows}
	In this section, we recall the basics of flows. This section is based on the introductory sections of \cite{fayad2002}, \cite{avila-forni-ulcigrai:mixing-heisenberg} and \cite{avila-forni-ulcigrai-ravotti:general-nilflows}.
	\subsubsection{Definitions}
	Let us assume $X$ is a metric space.
	\begin{definition}[flow]
		A flow $\{h_{t}\}_{t\in\R}$ on $X$ is a measurable action $h\colon X\times \R \to X$, denoted $h_{t}(x)=h(x,t)$, that satisfies the following group composition property \[
		h_{s+t}=h_{s}\circ h_{t},
		\]for every $s,t\in \R$ and such that $h_{0}=\id_{X}$.
	\end{definition}
	When $h$ is continuous, we refer to the flow as continuous. If, moreover, $X$ is a smooth manifold and $h$ is a smooth map, we call the flow smooth. In this smooth setting, the flow is determined by a vector field, known as the \emph{infinitesimal generator} of $h$. Concretely, $h$ moves points along the integral curves of $V$.
	\begin{definition}[infinitesimal generator]
		Let $\{h_{t}\}_{t\in\R}$ be a smooth flow on a manifold $X$. Then $V(x)=\left.\dfrac{\partial h_{t}(x)}{\partial t}\right\vert_{t=0}$ is called the infinitesimal generator of $\{h_{t}\}_{t\in\R}$.
	\end{definition}
	\begin{remark}
		Let us recall that for any vector field $V$ on a manifold $X$ we have the action of $V$ on the space of smooth functions $g\colon X\to \R$ on $X$, defined by \[(Vg)(x):=\left.\frac{\diff \left(g\circ \psi^{V}_{t}(x)\right)}{\diff t}\right\vert_{t=0},\] for $x\in X$, where $\psi^{V}_{t}$ is the flow along $V$ and $(-\varepsilon,\varepsilon)\ni s\mapsto\psi^{V}_{s}(x)$ is an integral curve of $V$ passing through $x$, in general defined only in a neighborhood of $s=0$, that is, for small $\varepsilon>0$. In other words, $Vg$ is the directional derivative of $g$ in the direction of $V(x)=\left.\frac{\diff \psi^{V}_{t}(x)}{\diff t}\right\vert_{t=0}$. Note that, formally, $Vg\colon X\to T \R$. We identify $T\R$ with $\R$.
	\end{remark}
	The main property investigated in the work of Kanigowski and Ravotti \cite{kanigowski2024multiplemixingparabolicsystems} and the authors' \cite{kucharski:multmixing} is the following. We will say that $\tb$ is a $k$-tuple of times, for $k\geq 2$, if $\tb=(t_{0},...,t_{k-1})\in (\R^{+})^{k}$ and $t_{0}=0< t_{1}<...<t_{k-1}$. We denote $\Delta\tb=\min_{i=1,...,k-1}|t_{i}-t_{i-1}|$.
	\begin{definition}[k-mixing]
		Let $k\geq 2$. We say that $\{h_{t}\}_{t\in\R}$ is \emph{k-mixing} if for every $k$-tuple of measurable subsets $A_{i}\subset X$, $i=0,...,k-1$, we have 
		$$
		\mu\left( \bigcap_{i=0}^{k-1}h_{t_{i}}\left(A_{i}\right)\right)\to\prod_{i=0}^{k-1}\mu\left(A_{i}\right)\,.
		$$where the limit is taken over $k$-tuple of times $\tb$ with $\Delta\tb\to\infty$.
	\end{definition}
	\begin{remark}
		It is well known that the flow $\{h_{t}\}_{t\in\R}$ is \emph{k-mixing} if and only if for every $k$-tuple of measurable functions $f_{i}\in L^{2}(X)$, $i=0,...,k-1$, we have 
		$$
		\int \prod_{i=0}^{k-1}f_{i}\circ h_{t_{i}}\diff \mu\to\prod_{i=0}^{k-1}\int f_{i}\diff \mu\,.
		$$where the limit is taken over $k$-tuples of times $\tb$ with $\Delta\tb\to\infty$.
	\end{remark}
	
	\subsection{Special flows}
	The present section is based on the introductory sections of \cite{fayad2002}, \cite{avila-forni-ulcigrai:mixing-heisenberg} and \cite{avila-forni-ulcigrai-ravotti:general-nilflows}. We recall the definition of a special flow.
	\subsubsection{Definition}
	Assume $X$ is a manifold and $\Sigma\subset X$ is a measurable subset. Let $\Phi\in L^{1}(\Sigma)$ be a function with $\Phi>c$, for some $c>0$, and $f\colon \Sigma\to\Sigma$ a measurable transformation. Let us recall that the $n$-th Birkhoff sum at $z$ along the orbit of $f\colon \Sigma\to\Sigma$ for $\Phi$ is defined as \[
	S_{n}(\Phi)(z)=\sum_{i=0}^{n-1}\Phi\circ f^{i}(z).
	\]A special flow over the base transformation $T$ with the roof function $\Phi$ is the quotient flow $f^{\Phi}=\{f^{\Phi}_{t}\}_{t\in\R}$ of the action \[
	\Sigma\times \R\longrightarrow \Sigma\times \R,
	\]\[
	(z,s)\mapsto (z,s+t),
	\]by the relation $(z,s+\Phi(z))\sim (f(z),s)$. The space $\Sigma$ is called \emph{the base} of $f^{\Phi}$, $f$ is \emph{the base transformation}, and $\Phi$ \emph{the roof function}. The flow $f^{\Phi}$ acts on the manifold $M_{f,\Phi}:=\Sigma\times \R/_{\sim}$, and so can be seen as defined on the fundamental domain $\{(z,s)\colon z\in \Sigma,~0\leq s\leq \Phi(z)\}$, which is the region under the graph of $\Phi$. The action is given explicitly by the formula \[
	f^{\Phi}_{t}(z,s)=(f^{N(z,t)}(z), s+t-S_{N(z,t+s)}(\Phi)(z)),
	\]where $N(z,t')=\max\{n\in\N\colon S_{n}(\Phi)(z)<t'\}$, for any $t'>0$. The action of flow $f^{\Phi}$ can be seen as the translation of $(z,0)$, $z\in \Sigma$, with unit speed along the vertical $s$-fibers. As $f^{\Phi}_{t}(z)$ reaches the roof function, for some $t>0$, we have the identification $f^{\Phi}_{t}(z)=(z,\Phi(z))\sim (f(z),0)$. The number $N(z,t)$ is essentially the cardinality of times $z$ has reached the roof function while being translated by $f^{\Phi}_{t}$. In other words, for $0\leq s+s'\leq \Phi(f^{n}(z))$ and $n:=N(z,t)$, we have $f^{\Phi}_{s+S_{n}(\Phi)(z)}(z,s')=(f^{n}(z),s)$.
	
	Note that for any $f$-invariant measure $\mu$ on $\Sigma$, the finite measure obtained by the restriction of the product measure $\mu\times\Leb$, where $\Leb$ is the Lebesgue measure on the $s$-fiber, to the domain of $f^{\Phi}$ is invariant by the special flow $f^{\Phi}$.
	
	\section{Quantifying the mechanism of shearing}
	\subsection{Locally uniformly shearing flows}
	The results of \cite{kanigowski2024multiplemixingparabolicsystems} introduced \emph{locally uniformly shearing flows}, abbreviated LUS flows, and proved that mixing implies mixing of any order for LUS flows. Below we enclose the definition of LUS flows and preliminary definitons. Let $\flowR$ be a flow on a measurable space $X$.
	
	\begin{definition}[{\cite[Def. 3.1]{kanigowski2024multiplemixingparabolicsystems}}]\label{def:almostmp}
		A measurable map $\mapg \colon (X, \mu) \to (Y,\nu)$ between two measure spaces $ (X, \mu)$ and $ (Y,\nu)$ is called \emph{$\varepsilon$-almost measure preserving} if there exists $X_0 \subseteq X$ and $ Y_0 \subseteq Y$ with $\mu(X\setminus X_0) < \varepsilon \mu(X)$ and $\nu(Y \setminus Y_0) < \varepsilon \nu(Y)$ such that $(1-\varepsilon) \nu(A) \leq \mu(\mapg^{-1}(A)\cap X_0) \leq (1+\varepsilon) \nu(A)$ for all measurable sets $A \subseteq Y_0$. We call $X_0$ and $Y_0$ the \emph{proper domain} and the  \emph{proper codomain} of $\mapg$ respectively.
	\end{definition}
	
	\begin{definition}[{\cite[Def. 5.1]{kanigowski2024multiplemixingparabolicsystems}}]\label{def:tower} 
		Let $B $ be measurable and $\height \geq 0$.
		The set $\cT=\bigcup_{0\leq t\leq \height}\flow_t(B)$ is called a {\emph tower with base $B$ and height $\height$} if, for any $t,t'\in [0,\height]$ with $t \neq t'$, we have $\flow_t(B)\cap \flow_{t'}(B)=\emptyset$. 
	\end{definition}
	
	\begin{definition}[{LUS flows, \cite[Def. 6.1]{kanigowski2024multiplemixingparabolicsystems}}] \label{def:LUS_flows}
		The flow $\flowR$ is called {\em locally uniformly shearing} if there exist $a \in (0,1/2]$, $d \geq 1$, an increasing sequence $(M_s)_{s \in \N}$, and, for every $\delta >0$ and $k \in \Z_{\geq 0}$, there exists $M_{k,\delta} > \delta^{-2}$ such that for every $M_s, \height \geq M_{k,\delta}$ there exist $t_{\height, M_s}\geq \max\{M_s, \height\}$ and a $\delta$-almost partition $\{\cR_{\varrho}\}_{\varrho =1}^d$ of $X$ into towers $\cR_{\varrho}$ of heights $\height_{\varrho} \in [\height^{1/2},\height^{3/2}]$ such that, for every $(k+1)$-tuple of times $\tb$ with $\Deltat > t_{\height, M_s}$, the following conditions hold:
		\begin{enumerate}
			\item[(LUS1)] there exist $\delta$-almost measure preserving maps $\mapg^{\tb}_m \colon \cup_{\varrho=1}^d \cR_{\varrho} \to \cup_{\varrho=1}^d \cR_{\varrho}$, for $m=0,\dots, K_{\delta} = \lfloor \delta^{-1}\rfloor$, with the following properties
			\begin{enumerate}
				\item for all $x$ 
				in the proper domain of $\mapg^{\tb}_m$, we have $\dist(\mapg^{\tb}_m(x),x)\leq \delta$,
				\item for any tower $\cT$ of height $H \geq \height^{1/2}$, there exists a tower $\overline{\cT} \subseteq (\mapg^{\tb}_m)^{-1}(\cT)$ of height $\height^{1/4}$ with
				\[
				\mu \left( (\mapg^{\tb}_m)^{-1}(\cT) \setminus \overline{\cT} \right) \leq d \,\delta;
				\]
			\end{enumerate} 
			\item[(LUS2)] for every $\varrho= 1,\dots, d$, there exists a $\delta$-almost partition $\{\Tow_{\varrho, j}\}_{j=1}^k$ of $\cR_{\varrho}$ into towers $\Tow_{\varrho, j} \subseteq \cR_{\varrho}$ of the same height $\height_{\varrho}$, so that, for every $m=0,\dots,K_{\delta}$, we have 
			\[
			\max_{j=1, \dots, k} \sup_{x\in \Tow_{\varrho,j}} \dist \left(\flow_{t_j}(\mapg^{\tb}_m x), \flow_{t_j+m \cdot M_s^a }(x) \right)\leq \delta;
			\]
			\item[(LUS3)] for every $\varrho=1, \dots, d$, there exists a $\delta$-almost partition of $\cR_{\varrho}$ into (possibly empty) towers $R_{\varrho}(i,m,\ell) \subseteq \cR_{\varrho}$, with $\ell = 0, \dots, \delta^{-1}M-1$, of the same height $\height_{\varrho}$ such that, for any $i=1, \dots, k$, any $m=0,\dots,K_{\varepsilon}$, and any $\ell$ as above, we have
			\[
			\sup_{x\in R_\varrho(i,m,\ell)} \dist \left(\flow_{t_i}(\mapg^{\tb}_m x), \flow_{t_i+ \ell \cdot \delta}(x) \right)\leq \delta.
			\]
		\end{enumerate}
		The main result of \cite{kanigowski2024multiplemixingparabolicsystems} reads as follows.
		\begin{proposition}[{\cite[Theorem 6.8]{kanigowski2024multiplemixingparabolicsystems}}]\label{proposition:mixingLUS->mixing-of-all-orders}
			Mixing LUS flows are mixing of all orders.
		\end{proposition}
	\end{definition}
	
	\subsection{Good special flows}We assume that the flow $\flowR$ is a special flow over $T \colon (X, \mathcal{B}, \mu) \to (X, \mathcal{B}, \mu)$ with roof function $\roof \in L^p(\baseX, \mu)$ for some $p >1$. Good special flows, see Def. \ref{def:good_sf}, were introduced by Kanigowski and Ravotti as a preliminary step in the proof that special flows over $\T$ satisfying certain conditions on the uniform distribution and shear of Birkhoff sums of the roof function are locally uniformly shearing flows. Below we present the definitions that can be found in \cite{kanigowski2024multiplemixingparabolicsystems}.
	\begin{definition}\label{definition:disintegration}
		We will say that $(X, \mathcal{B}, \mu)$ has disintegration $\{(X_{z}, \mathcal{B}_{z}, \mu_{z})\}_{z\in Z}$ if $(X_{z}, \mathcal{B}_{z}, \mu_{z})$ are measure spaces, $\{X_{z}\}_{z\in Z}$ is a partition of $X$ into measurable sets	$\{X_z \colon z \in Z \}$, the measures $\mu_z$ satisfy $\mu_z(X \setminus X_z) = 0$, and there exists a measure $\sigma$ on $Z$ such that for any measurable set $A \subseteq X$ we can write 
		\[
		\mu(A) = \int_Z \mu_z(A) \diff \sigma(z).
		\]
	\end{definition}
	We will usually abbreviate and say that $\mu$ on $X$ has disintegration
	$\{\mu_{z}\}_{z\in Z}$. Unless otherwise stated, all spaces are Lebesgue spaces.
	\begin{definition}[Good almost partitions]
		Let $\varepsilon \geq 0$ and let $A \subseteq X$ be measurable. A family $\cP=\{P_a\}_{a\in \mathscr{A}}$ of measurable sets $P_a \subseteq X$ is a \emph{good $\varepsilon$-almost partition of $A$} with respect to the disintegration $\{\mu_{z}\}_{z\in Z}$ if 
		\begin{enumerate}
			\item $\cP$ is a $\varepsilon$-almost partition of $A$ (with respect to the measure $\mu$ on $X$),
			\item for each $a\in \mathscr{A}$, there exists $z \in Z$ such that $P_a \subseteq X_z$,
			\item for every $z \in Z$, the collection $\{P_a \in \cP \colon P_a \subseteq X_z\}$ is at most countable.
		\end{enumerate}
	\end{definition}
	
	\begin{definition}[Good special flows]\label{def:good_sf}
		Let $\{\mu_{z}\}_{z\in Z}$ be a disintegration of $\mu$. The flow $\flowR$ is a {\em good special flow} if there exist $d \geq 1$, an increasing sequence $(M_s)_{s \in \N}$, and, for every $\varepsilon >0$ and $k \in \Z_{\geq 0}$, there exists $M_{k,\varepsilon} > \varepsilon^{-2005k}$ such that for every $M_s, \height \geq M_{k,\varepsilon}$ there exist $t_{\height, M_s}\geq \max\{M_s, \height\}$  and a $\varepsilon$-almost partition $\{\cR_{\varrho}\}_{\varrho=1}^d$ of $\susp$ into $d$ towers $\cR_{\varrho}$ of base $B_{\varrho} \subseteq X$ and heights $\height_{\varrho} \in [\height^{1/2},\height^{3/2}]$ for which the following holds.
		For every 
		$(k+1)$-tuple of times $\tb$ with $\Deltat > t_{\height, M_s}$, there exists a (possibly uncountable) family $\cP_{\tb}=\cP_{\tb}(\height,M)=\{P_a\}_{a\in \mathscr{A}}$ such that

		\begin{enumerate}
			\item[(G1)] we have
			\[
			\max_{i=0,\ldots, k}\sup_{a\in \mathscr{A}} \;  \sup_{x\in P_a} \; \max_{0 \leq t \leq M_s + \height^{3/2}} \diam( \flow_t( T^{N(x,t_i)} P_a )) < \varepsilon;
			\]
			\item[(G2)] $\cP_{\tb}$ is a good $\varepsilon$-almost partition with respect to $\{\mu_{z}\}_{z\in Z}$ of $B_\varrho$ for every $\varrho= 1,\dots, d$;
			\item[(G3)] for any $a\in \mathscr{A}$, $i \in \{1,\dots, k\}$, and $x \in P_a$, define $S^{+}(a,t_i,x)\geq 0$ (respectively $S^{-}(a,t_i,x)\leq 0$) be the smallest (respectively largest) number such that 
			\[
			\flow_{t_i}(P_a)\subset \bigcup_{S^{-}(a,t_i,x)\leq r\leq S^{+}(a,t_i,x)} \flow_r(T^{N(x,t_i)}(P_a)).
			\]
			Then, the latter set is a tower and $S^{+}(a,t_i,x)-S^{-}(a,t_i,x) <M_s$.
			
			\item[(G4)] For every $a\in \mathscr{A}$ there exists $i_a\in \{1,\ldots, k\}$ and $x_a\in P_a\subseteq X_z$ such that 
			$\flow_{t_{i_a}}(x_a,0)=(T^{N(x_a,t_{i_a})}x_a,0)$ and 
			$$
			\Delta S_a := S^{+}(a,t_{i_a},x_a)-S^{-}(a,t_{i_a},x_a)> \varepsilon^{1000k}M_s
			$$
			Moreover
			\begin{equation*}
				\frac{\varepsilon (1-\varepsilon) \mu_z(P_a)}{\Delta S_a } \leq  \mu_z\Big\{x\in P_a \, : \, \flow_{t_{i_a}}(x)\in \bigcup_{r\in [L,L+\varepsilon]}\flow_r(T^{N(x_a,t_{i_a})}P_a) \Big\} \leq  \frac{\varepsilon (1+\varepsilon) \mu_z(P_a)}{\Delta S_a },
			\end{equation*}
			for every $S^{-}(a,t_{i_a},x_a)\leq L\leq S^{+}(a,t_{i_a},x_a)-\varepsilon$.
		\end{enumerate}
	\end{definition}
	\begin{proposition}[{\cite[Proposition 8.5]{kanigowski2024multiplemixingparabolicsystems}}]\label{proposition:goodspecialflows->lusKanigRav}
		Good special flows are LUS.
	\end{proposition}
	The definition of good special flows in \cite{kanigowski2024multiplemixingparabolicsystems} was formulated for a fixed disintegration of $\mu$. As in \cite{kucharski:multmixing}, the disintegration will depend on $(k+1)$-tuple of times, we ought to adjust the definition of \cite{kanigowski2024multiplemixingparabolicsystems} and show that good special flows with this modified definition are still LUS. The flows satisfying the modified definition will be called \emph{good special flows with varying disintegration}. Moreover, the claim that good special flows with varying disintegration are LUS flows have analogous proof to the case of fixed disintegration; see \cite[Proposition 8.5]{kanigowski2024multiplemixingparabolicsystems} (see Proposition \ref{proposition:goodspecialflows->lusKanigRav}).
	
	\begin{definition}[Good special flows with varying disintegration]\label{definition:goodspecialsflows}
		The flow $\flowR$ is a {\em good special flow with varying disintegration} if there exist $d \geq 1$, an increasing sequence $(M_s)_{s \in \N}$, and, for every $\varepsilon >0$ and $k \in \Z_{\geq 0}$, there exists $M_{k,\varepsilon} > \varepsilon^{-2005k}$ such that for every $M_s, \height \geq M_{k,\varepsilon}$ there exist $t_{\height, M_s}\geq \max\{M_s, \height\}$  and a $\varepsilon$-almost partition $\{\cR_{\varrho}\}_{\varrho=1}^d$ of $\susp$ into $d$ towers $\cR_{\varrho}$ of base $B_{\varrho} \subseteq X$ and heights $\height_{\varrho} \in [\height^{1/2},\height^{3/2}]$ for which the following holds.
		For every 
		$(k+1)$-tuple of times $\tb$ with $\Deltat > t_{\height, M_s}$, there exists a (possibly uncountable) family $\cP_{\tb}=\cP_{\tb}(\height,M)=\{P_a\}_{a\in \mathscr{A}}$ and a disintegration of $\mu$ into $\{\mu_{z}\}_{z\in Z}$ such that
		
		\begin{enumerate}[label=(G'\arabic*)]
			\item\label{definition:goodspecialsflows:g1} we have
			\[
			\max_{i=0,\ldots, k}\sup_{a\in \mathscr{A}} \;  \sup_{x\in P_a} \; \max_{0 \leq t \leq M_s + \height^{3/2}} \diam( \flow_t( T^{N(x,t_i)} P_a )) < \varepsilon;
			\]
			\item\label{definition:goodspecialsflows:g2} $\cP_{\tb}$ is a good $\varepsilon$-almost partition of $B_\varrho$ with respect to $\{\mu_{z}\}_{z\in Z}$ for every $\varrho= 1,\dots, d$;
			\item\label{definition:goodspecialsflows:g3} for any $a\in \mathscr{A}$, $i \in \{1,\dots, k\}$, and $x \in P_a$, define $S^+=S^{+}(P_a,t_i,x)\geq 0$ (respectively $S^-=S^{-}(P_a,t_i,x)\leq 0$) to be the smallest (respectively largest) number such that 
			\[
			\flow_{t_i}(P_a)\subset \bigcup_{S^{-}\leq r\leq S^{+}} \flow_r(T^{N(x,t_i)}(P_a)).
			\]
			Then, the latter set is a tower and $S^{+}-S^{-} <M_s$.
			
			\item\label{definition:goodspecialsflows:g4} For every $a\in \mathscr{A}$ there exists $i_a\in \{1,\ldots, k\}$ and $x_a\in P_a\subseteq X_z$ such that 
			$\flow_{t_{i_a}}(x_a,0)=(T^{N(x_a,t_{i_a})}x_a,0)$ and 
			$$
			\Delta S_{P_a} := S^{+}(P_a,t_{i_a},x_a)-S^{-}(P_a,t_{i_a},x_a)> \varepsilon^{1000k}M_s
			$$
			Moreover
			\begin{equation*}
				\frac{\varepsilon (1-\varepsilon) \mu_z(P_a)}{\Delta S_a } \leq  \mu_z\Big\{x\in P_a \colon \flow_{t_{i_a}}(x)\in \bigcup_{r\in [L,L+\varepsilon]}\flow_r(T^{N(x_a,t_{i_a})}P_a) \Big\} \leq  \frac{\varepsilon (1+\varepsilon) \mu_z(P_a)}{\Delta S_a },
			\end{equation*}
			for every $S^{-}(P_a,t_{i_a},x_a)\leq L\leq S^{+}(P_a,t_{i_a},x_a)-\varepsilon$.
		\end{enumerate}
	\end{definition}
	Note that conditions \ref{definition:goodspecialsflows:g1}, \ref{definition:goodspecialsflows:g3} and \ref{definition:goodspecialsflows:g4} in the definitions of good special flows and good special flows with varying disintegration are the same. The only difference is that the latter definition allows for the disintegration to change with $(k+1)$-tuple of times. The main claim of this work is the following.
	\begin{proposition}[{Theorem \ref{thmmain:goodspecialflows->lus}}]\label{proposition:goodspecialflows->lus}
		Good special flows with varying disintegration are LUS.
	\end{proposition}
		\subsection*{A sketch of proof}Note that the almost partitions $\cP_{\tb}$ from the definition of good special flows depend on all the parameters introduced in the definition. In particular, the almost partitions depend on the $(k+1)$-tuple of times $\tb$. In the proof of Proposition 8.5 in \cite{kanigowski2024multiplemixingparabolicsystems} the authors, after fixing the parameters from the definition of good special flows, construct maps $\mapg^{\tb}_m$ that appear in the definition of LUS flows using the almost partitions $\cP_{\tb}$. That is, the construction of $\mapg^{\tb}_m$ is based on fixed parameters appearing in the definition of good special flows, and does not depend on the relative dependence of the parameters and almost partition themselves. In particular, the construction presented in the proof of Proposition 8.5 in \cite{kanigowski2024multiplemixingparabolicsystems} is valid regardless of whether the disintegration is fixed beforehand or is dependent on the $(k+1)$-tuple of times. Thus, the said construction of $\mapg^{\tb}_m$ can be used to show that good special flows with varying disintegration are LUS. 
	\subsection{Proof of \Cref{thmmain:goodspecialflows->lus}}\label{sec:contruct_mapg}

We now show that good special flows with varying disintegration are LUS.

Let $\flowR$ be a good special flow, and let $d \geq 1$ and $(M_s)_{s \in \N}$ be given by definition.
In \Cref{def:LUS_flows}, we set the same $d$ and $(M_s)_{s \in \N}$, together with $a=1/2$. 
Fix $\delta >0$ and $k \in \Z_{\geq 0}$; we let $\varepsilon = \delta^4/(20d)$ and $ M_{k, \varepsilon} \geq \varepsilon^{-2005k}$ be given by the definition of good special flows with varying disintegration.
Fix $M_s, \height \geq M_{k, \varepsilon}$. Let $(\cR_\varrho)_{\varrho =1}^d$ be Rokhlin towers of heights $\height_\varrho \in [\height^{1/2},\height^{3/2}]$ and bases $B_\varrho \subset X$ such that $\mu(\susp \setminus \cup_\varrho \cR_\varrho) \leq \varepsilon$. 
Let us fix a $(k+1)$-tuple of times $\tb$ with $\Deltat > t_{\height, M_s}$, where $t_{\height, M_s}$ is given by \Cref{definition:goodspecialsflows}. Let also $\cP_{\tb} = \{P_a\}_{a \in \mathscr{A}}$ be the associated good $\varepsilon$-almost partition of the bases $B_{\varrho}$ with respect to the disintegration $\{\mu_z\}_{z\in Z}$.

Now note, that the almost partitions $\cP_{\tb}$ and the corresponding disintegration of $\mu$ depend on all other parameters of the definition of good special flows with varying disintegration. Therefore, once the parameters are fixed, and we are presented with the almost partition $\cP_{\tb}$, we are essentially in the same position as in the definition of good special flows with a fixed disintegration. That is, it does not matter, whether the disintegration is fixed and the parameters depend on it, or the other way around, as the construction is carried out once both the parameters, the almost partitions and the disintegration is fixed.

It follows from this discussion that we can construct $\delta^2$-almost measure preserving maps on the bases $B_\varrho$ just as in the proof of Kanigowski and Ravotti's claim that good special flows are LUS. For the convenience of the reader we enclose the proof as it stands in \cite{kanigowski2024multiplemixingparabolicsystems}. We do not change the exposition of the proof of Kanigowski and Ravotti, and so with the next paragraph starts the proof as it can be found in \cite{kanigowski2024multiplemixingparabolicsystems} and extends till the end of this note. We stress that all the credit for the proof belongs to Kanigowski and Ravotti.

Let us construct $\delta^2$-almost measure preserving maps on the bases $B_\varrho$ with the aid of $\cP_{\tb}$, and then extend them to the union of the towers $\cR_\varrho$.

Fix $a \in \mathscr{A}$ and consider $i_a \in \{1,\dots, k\}$, $x_a \in P_a$, and $S_a^{\pm} = S_a^{\pm}(a,t_{i_a},x_a)$ be given by \ref{definition:goodspecialsflows:g4} in \Cref{def:good_sf}.
For $0\leq L \leq \Delta S_a - \varepsilon$, where $\Delta S_a := S_a^{+} - S_a^{-} $, we define
\[
A_L^{a,\tb} := \left\{ x \in P_a : \flow_{t_{i_a}}(x) \in \bigcup_{r\in [S_a^{-} + L,S_a^{-} + L+\varepsilon]}\flow_r(T^{N(x_a,t_{i_a})}P_a) \right\}.
\]
 
By property \ref{definition:goodspecialsflows:g4}, for all $L,L'$ as above, we have
\begin{equation}\label{eq:unifmeasureA}
1-3\varepsilon \leq \frac{\nu_z(A_L^{a,\tb})}{\nu_z(A_{L'}^{a,\tb})} \leq 1+3\varepsilon,
\end{equation}
provided that $\varepsilon <1/3$.

For $m \leq K_{\varepsilon} := \lfloor \varepsilon^{-1/4} \rfloor \leq \lfloor \delta^{-1} \rfloor$, let $N_m = m \cdot \sqrt{M_s}$. Let us also consider $0\leq \ell \leq \varepsilon^{-1} (\Delta S_a -  K_{\varepsilon} \sqrt{M_s})$ and $L = \ell \cdot \varepsilon \leq \Delta S_a -  K_{\varepsilon} \sqrt{M_s}$. By \eqref{eq:unifmeasureA}, there exists a $3\varepsilon$-almost measure preserving map
\[
\mapg^{a,\tb}_{m,L} \colon A_L^{a,\tb} \to A_{L+N_m}^{a,\tb}.
\]

\begin{lemma}
There exists a $3\varepsilon$-almost measure preserving map $\mapg^{a,\tb}_m \colon P_a \to P_a$ whose restriction to any $A_{\ell \varepsilon}^{a,\tb}$ with $0 \leq \ell \leq \varepsilon^{-1} (\Delta S_a - K_{\varepsilon}\sqrt{M_s})$ coincides with $\mapg^{a,\tb}_{m,L}$.
\end{lemma}
\begin{proof}
It suffices to show that the sets $\{A_{\ell \varepsilon}^{a,\tb} : 0\leq \ell \leq \varepsilon^{-1} (\Delta S - K_{\varepsilon}\sqrt{M_s})\}$ are pairwise disjoint and they cover $P_a$ up to a set of measure at most $\varepsilon^2 \nu_z(P_a)$; their union $D_m^{a,\tb}$ will form the proper domain of $\mapg^{a,\tb}_m$.

The fact that the sets $A_{\ell \varepsilon}^{a,\tb}$ are disjoint for different $\ell$ follows from their definition and property \ref{definition:goodspecialsflows:g3}. Using property \ref{definition:goodspecialsflows:g4} we have 
\begin{multline*}
\nu_z \left(P_a \setminus \bigcup\{A_{\ell \varepsilon}^{a,\tb} : 0\leq \ell \leq \varepsilon^{-1} (\Delta S_a -  K_{\varepsilon} \sqrt{M_s})\} \right) \leq \sum_{\ell =  \varepsilon^{-1} (\Delta S_a -  K_{\varepsilon} \sqrt{M_s})}^{ \varepsilon^{-1} (\Delta S_a - \varepsilon)} \nu_z(A_{\ell \varepsilon}^{a,\tb})  \\ \leq  \varepsilon^{-1}  K_{\varepsilon} \sqrt{M_s}  \frac{\varepsilon (1+\varepsilon) \nu_z(P_a)}{\Delta S_a} \leq 2 \varepsilon^{-1/4 - 1000k} M_{k,\varepsilon}^{-1/2} \nu_z(P_a) \leq \varepsilon^2 \nu_z(P_a),
\end{multline*}
since $M_{k,\varepsilon} \geq \varepsilon^{-2005k}$. 
\end{proof}

We now extend $\mapg^{a,\tb}_m$ to almost measure preserving maps over the bases $B_\varrho$ as follows.

\begin{lemma}\label{lem:mapg_on_bases}
	For any $\varrho=1, \dots, d$, there exists a $3\sqrt{\varepsilon}$-almost measure preserving map $\mapg^{\varrho,\tb}_m \colon B_\varrho \to B_\varrho$ whose restriction to any $P_a$ coincides with $\mapg^{a,\tb}_m$.
\end{lemma}
\begin{proof}
	We can define $\mapg^{\tb}_m \colon \cup \cP_{\tb} \to \cup \cP_{\tb}$ by $\mapg^{a,\tb}_m$ for each atom of the almost partition $\cP_{\tb}$. The proper domain of this map is then the union $ \cup_a D_m^{a, \tb}$ of the proper domains $D_m^{a, \tb}$ of the maps $\mapg^{a,\tb}_m$. Let us show that it restrict to a $3\sqrt{\varepsilon}$-almost measure preserving map of the bases $B_\varrho$ of each Rokhlin tower $\cR_\varrho$. By \ref{definition:goodspecialsflows:g2}, we have 
	\[
	\nu\left(B_\varrho \triangle \bigcup \{P_a \in \cP_{\tb} : P_a \cap B_\varrho \neq \emptyset \} \right)<\varepsilon \nu(B_\varrho).  
	\]
	Thus, using the disintegration of $\nu$, we get
	\[
	\begin{split}
		\varepsilon \nu(B_\varrho) &\geq \nu \left(B_\varrho^c \cap \bigcup \{P_a \in \cP_{\tb} : P_a \cap B_\varrho \neq \emptyset \} \right) \\
		&\geq \int_Z \nu_z \left(B_\varrho^c \cap \bigcup \{ P_a \subseteq X_z : P_a \cap B_\varrho \neq \emptyset \text{ and } \nu_z (B_\varrho^c \cap P_a) \geq \sqrt{\varepsilon} \nu_z(P_a)  \} \right) \diff \sigma(z)\\
		& = \int_Z \sum \left\{\nu_z (B_\varrho^c \cap  P_a) \ :\ P_a \subseteq X_z, \text{ } P_a \cap B_\varrho \neq \emptyset \text{ and } \nu_z (B_\varrho^c \cap P_a) \geq \sqrt{\varepsilon} \nu_z(P_a)\right\}\\
		& \geq \sqrt{\varepsilon}\nu\left(\bigcup \{ P_a : P_a \cap B_\varrho \neq \emptyset \text{ and } \nu_z (B_\varrho^c \cap P_a) \geq \sqrt{\varepsilon} \nu_z(P_a)  \} \right);
	\end{split}
	\]
	note that the sum above is at most countable by the properties of $\cP_{\tb}$. 
	We deduce that
	\[
	\nu\left(\bigcup \{ P_a : P_a \cap B_\varrho \neq \emptyset \text{ and } \nu_z (B_\varrho \cap P_a) < (1-\sqrt{\varepsilon}) \nu_z(P_a)  \} \right) \leq \sqrt{\varepsilon} \nu(B_\varrho);
	\]
	in other words, up to removing from $B_\varrho$ a set of atoms whose union has measure at most $\sqrt{\varepsilon} \nu(B_\varrho)$, we can make sure that all atoms $P_a$ which intersect $B_\varrho$ do so in a set of relative measure at least $(1-\sqrt{\varepsilon})$. Let us call $\mathscr{A}'$ the corresponding collection of indexes of these atoms. 
	Then, the set 
	\[
	D_m^{\varrho, \tb} = \bigcup_{a \in \mathscr{A}'} D_m^{a,\tb} \cap B_\varrho \cap (\mapg^{a,\tb}_m)^{-1} (B_\varrho)
	\]
	has measure at least $(1-3\sqrt{\varepsilon})\nu(B_\varrho)$ and is a proper domain for $\mapg^{\varrho,\tb}_m \colon B_\varrho \to B_\varrho$. This proves the lemma.
\end{proof}

By the properties of towers, we can extend the maps $\mapg^{\varrho,\tb}_m$ to the whole tower $\cR_\varrho$ by setting
\begin{equation}\label{eq:towdef}
\mapg^{\varrho,\tb}_m (\flow_t(x)) = \flow_t(\mapg^{\varrho,\tb}_m (x)),
\end{equation}
for $x \in B_\varrho$ and $0\leq t \leq \height_\varrho$. In particular, since the towers $\cR_\varrho$ are disjoint, we can define a map $\mapg^{\tb}_m \colon \susp \to \susp$ by $\mapg^{\tb}_m|_{\cR_\varrho} = \mapg^{\varrho,\tb}_m$, which is $3\sqrt{\varepsilon}$-almost measure preserving. We denote by 
\[
D^{\tb}= \bigcap_{m=0}^{ K_{\varepsilon} } \bigcup_{\varrho=1}^d \bigcup_{t \in [0,\height_\varrho]} \flow_t(D_m^{\varrho,\tb})
\]
the intersection of the proper domains of $\mapg^{\tb}_m$, for all $m=0, \dots, K_{\varepsilon}$.
Note that $\mu(\susp \setminus D^{\tb}) \leq 4\sqrt{\varepsilon}  K_{\varepsilon} \leq 4 \varepsilon^{1/4} \leq \delta$.

Let us verify (LUS1): fix $0 \leq m \leq K_{\varepsilon}$ and $x \in D^{\tb}$, write $x = \flow_t(y) \in \cR_{\varrho}$ for some $0\leq t \leq \height_{\varrho}$ and $y \in D_m^{\varrho,\tb} \subseteq P_a$. Then, by \ref{definition:goodspecialsflows:g1},
\[
\dist \left(\mapg^{\tb}_m (x), x \right)= \dist \left(\flow_t(\mapg^{\varrho, \tb}_m y), \flow_t(y) \right) \leq \diam(\flow_t(P_a)) \leq \varepsilon.
\]
This proves (LUS1)-(a); the following lemma proves (LUS1)-(b).

\begin{lemma}\label{lem:mapg_sends_towers_to_towers}
	Let 
	$\cT$ be a tower of height $H \geq \height^{1/2}$. There exists a tower $\overline{\cT} \subseteq (\mapg^{\tb}_m)^{-1}(\cT)$ of height $\height^{1/4}$ so that 
	\[
	\mu \left((\mapg^{\tb}_m)^{-1}(\cT) \setminus \overline{\cT}\right) \leq 16d \varepsilon^{1/4} \qquad \text{for all} \qquad m=0, \dots, K_{\varepsilon}.
	\]
\end{lemma}
\begin{proof}
	Since $H, \height_\varrho \geq \height^{1/2}$, by Lemma 5.5 in \cite{kanigowski2024multiplemixingparabolicsystems}, for every $\varrho = 1,\dots,d$, there exists a tower $\cT_{\varrho} \subseteq \cT \cap \cR_{\varrho}$ of height $\height^{1/4}$ so that $\mu(\cT \cap \cR_{\varrho} \setminus \cT_{\varrho})\leq 4 \height^{-1/4}$. By construction of the map $\mapg^{\tb}_m$, see \eqref{eq:towdef}, the set $(\mapg^{\tb}_m)^{-1}(\cT_{\varrho})$ is a tower contained in $\cR_{\varrho}$ of the same height as $\cT_{\varrho}$. We then consider the tower $\overline{\cT} = (\mapg^{\tb}_m)^{-1}(\cT_1) \cup \cdots \cup (\mapg^{\tb}_m)^{-1}(\cT_d)$. By the almost measure preserving property, we have
	\[
	\begin{split}
		\mu \left( (\mapg^{\tb}_m)^{-1}(\cT) \setminus \overline{\cT} \right) &\leq   \mu\left( (\mapg^{\tb}_m)^{-1}(\cT) \setminus \bigcup_{\varrho = 1}^d \cR_{\varrho}\right) + (1+\varepsilon) \sum_{\varrho = 1}^d \mu\left( \cT \cap \cR_{\varrho} \setminus \cT_{\varrho} \right) \\
		&\leq 5 \varepsilon^{1/4}+ 4d(1+\varepsilon)\height^{-1/4} \leq 16d \varepsilon^{1/4},
	\end{split}
	\]
	which completes the proof.
\end{proof}

We group the atoms $P_a$ (and hence the maps $\mapg^{\tb}_m$ as well) into the following sets: for $j \in \{1, \dots, k\}$, let $\mathscr{A}_{j} = \{ a \in \mathscr{A} : i_a = j\}$. We then define the corresponding bases and towers
\[
B_{\varrho,j} = \bigcap_{m=0}^{ K_{\varepsilon} }D_m^{\varrho,\tb} \cap \bigcup_{a\in \mathscr{A}_{j}} P_a \subseteq B_\varrho, \qquad \text{and} \qquad \Tow_{\varrho,j} = \bigcup_{0\leq t\leq \height_\varrho} \flow_t(B_{\varrho,j}) \subseteq D^{\tb} \cap \cR_\varrho.
\]
In other words, $B_{\varrho,j}$ consists of those points in the base $B_\varrho$ which belong to the proper domain of the maps $\mapg^{\varrho,\tb}_m$ for all $m=0,\dots, K_{\varepsilon}$ and moreover for which the corresponding $i_a$ in \ref{definition:goodspecialsflows:g4} equals $j$. Note that 
\begin{equation}\label{eq:measure_union_towj}
	\begin{split}
&\mu \left( \bigcup_{j=1}^k \Tow_{\varrho, j}\right) \geq \mu(\cR_\varrho) - \nu\left(B_\varrho \setminus \bigcup_{m=0}^{ K_{\varepsilon} }D_m^{\varrho,\tb}\right) \height_\varrho \geq \big( 1- 4\varepsilon^{1/4} \big)\mu(\cR_\varrho), \quad \text{and} \\
&\mu \left( \bigcup_{\varrho=1}^d \bigcup_{j=1}^k \Tow_{\varrho, j}\right) \geq 1 - \mu \left( \susp \setminus \bigcup_{\varrho=1}^d \cR_\varrho \right) - \sum_{\varrho=1}^{d} \mu \left( \cR_\varrho \setminus \bigcup_{j=1}^k \Tow_{\varrho, j}\right) \geq 1- 5\varepsilon^{1/4}.
	\end{split}
\end{equation}
We can now prove property (LUS2).

\begin{lemma}\label{lem:locren} Property (LUS2) is satisfied: for any $m\leq  K_{\varepsilon} $, $j  \in \{1, \dots, k\}$, and $\varrho \in \{1, \dots, d\}$ we have 
	\[
	\sup_{x\in \Tow_{\varrho,j}} \dist (\flow_{t_j}(\mapg^{\tb}_m x), \flow_{t_j+N_m}(x))\leq 3\varepsilon. 
	\]
\end{lemma}
\begin{proof} 
	Let us consider $x = \flow_H(y) \in \Tow_{\varrho,j}$, with $y\in B_{\varrho,j}$ and $H \leq \height_\varrho$. In particular, there exists $a \in \mathscr{A}_{j}$ as defined above so that $y \in P_a$, and $y \in D_m^{\varrho,\tb}$ for all $m = 0, \dots, K_{\varepsilon}$. By construction, there exists $0\leq \ell \leq \varepsilon^{-1} (\Delta S_a -  K_{\varepsilon} \sqrt{M_s})$ so that $y \in A^{a,\tb}_{\ell \varepsilon} \subset P_a$ and $\mapg^{\tb}_m y \in A^{a,\tb}_{\ell \varepsilon + N_m} \subset P_a$. Thus, there exist $r_1 \in [S_a^{-} + \ell \varepsilon, S_a^{-} + (\ell +1) \varepsilon]$ and $r_2 \in [S_a^{-} + \ell \varepsilon +N_m, S_a^{-} + (\ell +1) \varepsilon +N_m]$, where $S_a^{-} = S^{-}(a,t_{j},x_a)$, so that
	\[
	\flow_{t_{j}}(y) \in \flow_{r_1} (T^{N(x_a,t_{j})}P_a), \qquad \text{and} \qquad  \flow_{t_{j}}(\mapg^{\tb}_m y) \in \flow_{r_2} (T^{N(x_a,t_{j})}P_a).
	\]
	This, recalling the definition \eqref{eq:towdef}, implies that 
	\begin{multline*}
	\dist (\flow_{t_{j}}(\mapg^{\tb}_m x), \flow_{t_{j} + N_m}(x)) = \dist (\flow_{t_{j}+H}(\mapg^{\tb}_m y), \flow_{t_{j} + N_m + H}(y)) \\ \leq |r_2 - r_1 -N_m| + \diam (\flow_{r_2+H} (T^{N(x_a,t_{j})}P_a)) \leq 3\varepsilon,
	\end{multline*}
	where we used the fact that $r_2 + H \leq \Delta S_a + \height_\varrho \leq M_s + \height^{3/2}$, together with property \ref{definition:goodspecialsflows:g1}.
\end{proof}

We now define subtowers of the towers $\cR_\varrho$ on which we control the shear for the other times $t_i$, in order to prove (LUS3).
For any $a \in \mathscr{A}$, we fix $x_0 \in P_a$. Now, given $ i \in {1,\dots,k}$ and $0 \leq \ell \leq \varepsilon^{-1} M_s - 1$, we let 
\[
P_a(i,\ell) := \left\{ x \in P_a : \flow_{t_i}(x) \in \bigcup_{r\in [S_0^{-} + \varepsilon \ell,S_0^{-} + \varepsilon (\ell +1)]}\flow_r(T^{N(x_0,t_{i})}P_a) \right\},
\]
where $S_0^{-} = S^{-}(a,t_i,x_0)$.
Note that, by property \ref{definition:goodspecialsflows:g3}, the sets $P_a(i,0), \dots, P_a(i,\varepsilon^{-1} M_s - 1)$ form a partition of $P_a$ (note that the sets $P_a(i,\ell)$ for $\varepsilon^{-1}(S_0^+-S_0^-) \leq \ell \leq \varepsilon^{-1} M_s -1$ are actually empty). 

Given $0 \leq \ell \leq \varepsilon^{-1} M_s - 1$, we define
\[
Q_a(i, m, \ell) = \bigcup_{|\ell_2 - \ell_1| = \ell} P_a(i,\ell_1) \cap (\mapg^{\tb}_m)^{-1}(P_a(i,\ell_2)) \subseteq P_a, \qquad Q(i, m, \ell) = \bigcup_{a \in \mathscr{A}}Q_a(i, m, \ell),
\]
and
\[
R_\varrho(i, m, \ell) = \bigcup_{0\leq t \leq \height_\varrho} \flow_t(Q(i, m, \ell) \cap B_\varrho) \subseteq \cR_\varrho.
\]
\begin{lemma}\label{lem:R_varsigma_almost_partition}
For any fixed $\varrho\in \{1,\dots, d\}$, $i\in \{1,\dots, k\}$, and $m\in \{0,\dots, K_{\varepsilon}\}$, the sets $R_\varrho(i, m,0), \dots, R_\varrho(i, m,\varepsilon^{-1} M_s - 1)$ form a $4\varepsilon$-almost partition of $\cR_\varrho$.
\end{lemma}
\begin{proof}
	The sets $P_a(i,\ell)$ are pairwise disjoint for different values of $\ell$, hence so are the sets $Q_a(i,m,\ell)$. By the properties of towers, the sets $R_\varrho(i, m,0), \dots, R_\varrho(i, m,\varepsilon^{-1} M - 1)$ are also pairwise disjoint.
	
Since
\[
\bigcup_{\ell = 0}^{\varepsilon^{-1} M_s - 1} Q_a(i, m, \ell) = \bigcup_{\ell_1 = 0}^{\varepsilon^{-1} M_s - 1} P_a(i,\ell_1) \cap (\mapg^{\tb}_m)^{-1}\left(\bigcup_{\ell_2 = 0}^{\varepsilon^{-1} M_s - 1} P_a(i,\ell_2)\right) = P_a \cap (\mapg^{\tb}_m)^{-1}(P_a),
\]
the almost measure preserving property of $\mapg^{\tb}_m = \mapg^{a,\tb}_m$ on $P_a$ tells us that 
\[
\nu\left( \bigcup_{\ell = 0}^{\varepsilon^{-1} M_s - 1} Q_a(i, m, \ell)\right) \geq (1-3\varepsilon) \nu(P_a).
\]
Therefore, by property \ref{definition:goodspecialsflows:g2},
\[
\nu\left( B_\varrho \cap \bigcup_{a \in \mathscr{A}} \bigcup_{\ell = 0}^{\varepsilon^{-1} M_s - 1} Q_a(i, m, \ell)\right) \geq \nu\left( B_\varrho \cap \bigcup_{a \in \mathscr{A}} P_a \right) - 3\varepsilon \nu(B_\varrho) \geq (1-4\varepsilon) \nu(B_\varrho),
\]
and thus
\[
\mu\left(\bigcup_{\ell = 0}^{\varepsilon^{-1} M_s - 1} R_\varrho(i, m, \ell)\right) = \height_\varrho \cdot \nu\left( B_\varrho \cap \bigcup_{a \in \mathscr{A}} \bigcup_{\ell = 0}^{\varepsilon^{-1} M_s - 1} Q_a(i, m, \ell) \right) \geq (1-4\varepsilon) \mu(\cR_\varrho),
\]
which proves the claim.
\end{proof}

On the towers $R_\varrho(i,m,\ell)$, the control the shear is as dictated by (LUS3), which completes the proof of \Cref{proposition:goodspecialflows->lus}.
\begin{lemma}\label{lem:disinter} 
	Property (LUS3) holds: for every $i \in \{1, \dots, k\}$, every $0 \leq \ell \leq \varepsilon^{-1} M - 1$, every $m \leq  K_{\varepsilon} $, and every $x\in R_\varrho(i,m, \ell)$, we have
	\[
		\dist \Big(\flow_{t_i}(\mapg^{\tb}_m x),f_{t_i+\varepsilon \ell}(x)\Big)<2 \varepsilon.
	\]
\end{lemma}
\begin{proof}
	Let $x \in R_\varrho(i,m, \ell)$, and write $x = \flow_H(y)$, with $y \in B_\varrho \cap P_a(i,\ell_1) \cap (\mapg^{\tb}_m)^{-1}(P_a(i,\ell_2))$ for some $a \in \mathscr{A}$, some $|\ell_2 - \ell_1| = \ell$, and $0\leq H \leq \height_\varrho$. Thus, there exist $0 \leq \delta_1, \delta_2 \leq \varepsilon$ so that $\flow_{t_i}(\mapg^{\tb}_m y)$ and $\flow_{t_i + \varepsilon \ell + \delta_1}(y)$ both belong to $\flow_{S^- + \varepsilon \ell_2 + \delta_2 }(T^{N(x_0,t_i)}P_a)$. Using property \ref{definition:goodspecialsflows:g1}, we conclude
	\[
	\dist \Big(\flow_{t_i}(\mapg^{\tb}_m x),f_{t_i+\varepsilon \ell}(x)\Big) \leq \varepsilon + \diam \left(\flow_{H+S^- + \varepsilon \ell_2 + \delta_2 }(T^{N(x_0,t_i)}P_a)\right) \leq 2\varepsilon.
	\]
\end{proof}

	\bibliography{multiple-mixing}
	\bibliographystyle{plain}
\end{document}